\documentclass[11pt,a4paper]{article}
\usepackage{amsmath}
\usepackage{amsfonts}
\usepackage{amssymb}
\usepackage{url}
\usepackage{mathtools}
\usepackage{amsthm}
\usepackage[shortlabels]{enumitem}
\theoremstyle{definition}

\newtheorem{theorem}{Theorem}[section]
\newtheorem{corollary}[theorem]{Corollary}
\newtheorem{conjecture}[theorem]{Conjecture}

\newtheorem{lemma}[theorem]{Lemma}
\newtheorem{remark}[theorem]{Remark}
\newcommand{\h}{\mathcal{H}}
\newcommand{\F}{\mathbb{F}}
\newcommand{\cF}{\mathcal{F}}
\newcommand{\PSL}{\mathrm{PSL}}
\newcommand{\PGL}{\mathrm{PGL}}
\newcommand{\PG}{\mathrm{PG}}
\newcommand{\GL}{\mathrm{GL}}
\newcommand{\SL}{\mathrm{SL}}
\usepackage{authblk}

\title{An Infinite Family of Connected 1-Factorisations of Complete 3-Uniform Hypergraphs}
\author{Barbara Maenhaut\thanks{bmm@maths.uq.edu.au} \, Jeremy Mitchell\thanks{jeremy.mitchell@uq.edu.au}\, and Anna Pusk\'as \thanks{anna.puskas@glasgow.ac.uk}}
\affil{{\small School of Mathematics and Physics\\ The University of Queensland\\ QLD 4072, Australia} }
\affil{{\small School of Mathematics \& Statistics \\ University of Glasgow \\ Glasgow, Scotland} }
\date{\today}

\begin{document}
\maketitle

\begin{abstract}
    A connected 1-factorisation is a 1-factorisation of a hypergraph for which the union of each pair of distinct 1-factors is a connected hypergraph. A uniform 1-factorisation is a 1-factorisation of a hypergraph for which the union of each pair of distinct 1-factors is isomorphic to the same subhypergraph, and a uniform-connected 1-factorisation is a uniform 1-factorisation in which that subhypergraph is connected. Chen and Lu [Journal of Algebraic Combinatorics, 46(2) 475--497, 2017] describe a family of 1-factorisations of the complete 3-uniform hypergraph on $q+1$ vertices, where $q\equiv 2\pmod 3$ is a prime power. In this paper, we show that their construction yields a connected 1-factorisation only when $q=2,5,11$ or $q=2^p$ for some odd prime $p$, and a uniform 1-factorisation only for $q=2,5,8$ (each of these is a uniform-connected 1-factorisation).

\end{abstract}

\section{Introduction}
A \emph{1-factor} of a graph $G$ is a spanning 1-regular subgraph of $G$, and a \emph{1-factorisation} of $G$ is a collection of edge-disjoint 1-factors of $G$ that partition the edge-set of $G$. It is natural to ask: under what conditions does the complete graph on $n$ vertices ($K_n$) admit a 1-factorisation? It is clear that $n$ must be even. By Kirkman's 1847 construction of 1-factorisations of $K_n$ for all even integers $n\geq 2$ \cite{kirkman1847problem}, this condition is sufficient.

Given a 1-factorisation of a graph $G$, a well-studied problem is to ask if the union of each pair of its 1-factors is isomorphic to the same subgraph $H$ of $G$.  Such a 1-factorisation is called a \emph{uniform 1-factorisation} (U1F) of $G$ and the subgraph $H$ is called the \emph{common graph}. Furthermore, a uniform 1-factorisation in which the common graph is a Hamilton cycle is called a \emph{perfect 1-factorisation} (P1F). In the 1960's, Kotzig \cite{kotzig1964theory} posed a question which has become known as Kotzig's perfect 1-factorisation conjecture, namely that for each even integer $n$, the complete graph $K_{n}$ admits a perfect 1-factorisation. Three infinite families of perfect 1-factorisations of complete graphs are known to exist, covering orders   $n=p+1$ and $n=2p$ where $p$ is an odd prime \cite{bryantp1f}. Perfect 1-factorisations of complete graphs are also known to exist for all even orders up to $n=56$ and some other sporadic orders, however the conjecture remains open. For an updated overview of the problem, we recommend a survey by Rosa \cite{alexander2019perfect} and a paper on the number of non-isomorphic P1Fs of $K_{16}$ by Gill and Wanless \cite{gill2020perfect}. Recently, Davies, Maenhaut, and Mitchell \cite{mitchell2022factorisation}, have generalised the notions of uniform and perfect 1-factorisations of graphs to the context of hypergraphs.

A {\em hypergraph} $\h$ consists of a non-empty vertex set $V(\h)$ and an edge set $E(\h)$ in which each element of $E(\h)$ is a non-empty subset of the vertex set $V(\h)$. The complete $k$-uniform hypergraph of order $n$, denoted $K_n^k$, is the hypergraph with $n$ vertices, in which the edges are precisely all the $k$-subsets of the vertex set. A spanning 1-regular subhypergraph of a hypergraph is known as a \textit{1-factor}. A partition of the edge set of a hypergraph $\h$ into 1-factors is called a \emph{1-factorisation} of $\h$, and such a 1-factorisation having $\alpha$ 1-factors is often denoted by $\mathcal{F}=\{F_1, \dots, F_\alpha\}$. An obvious necessary condition for the existence of a 1-factorisation of the complete $k$-uniform hypergraph on $n$ vertices is that $k|n$. Baranyai \cite{baranyai1974factorization} showed that for $k\geq 3$, this condition is also sufficient.

A \emph{path} between two vertices, $x$ and $y$, of a hypergraph $\h$ is an alternating sequence of vertices and edges
$$
    [x=v_1,e_1,v_2,e_2,\dots,v_s,e_s,v_{s+1}=y],
$$
such that $v_1,v_2,\dots,v_{s+1}$ are distinct vertices of $\h$, and $e_1,e_2,\dots e_s$ are distinct edges of $\h$ such that $\{v_i,v_{i+1}\}\subseteq e_i$ for $1\leq i \leq s$. If every two vertices of a hypergraph $\h$ have a path between them, we say that $\h$ is \emph{connected}. A \textit{Berge cycle} in a hypergraph $\h$, is an alternating sequence of vertices and edges
$$(v_1, e_1, v_2, e_2, \dots, v_m, e_m),$$
such that $[v_1, e_1, v_2, e_2, \dots, v_m]$ is a path in $\h$, $\{v_1,v_m\}\subseteq e_m,$ and $e_m\in E(\h)\setminus\{e_1, e_2, \dots, e_{m-1}\}$. Note that each edge $e_i$ may contain vertices other than $v_i$ and $v_{i+1}$ including vertices outside of $\{v_1,\dots,v_m\}$. A \textit{Hamilton Berge cycle} in a hypergraph $\h$ is a Berge cycle in $\h$ for which $\{v_1,\dots,v_m\}$ is the vertex set of $\h$.

In \cite{mitchell2022factorisation}, Davies et al.\ generalised uniform and perfect 1-factorisations of graphs to the context of hypergraphs in several different ways, leading to the following definitions. A \emph{connected 1-factorisation} (C1F) is a 1-factorisation of a hypergraph for which the union of each pair of distinct 1-factors is a connected hypergraph. A \emph{uniform 1-factorisation} (U1F) is a 1-factorisation of a hypergraph for which the union of each pair of 1-factors is isomorphic to the same subhypergraph, called the common hypergraph, and a \emph{uniform-connected 1-factorisation} (UC1F) is a U1F in which the common hypergraph is connected. A \emph{Hamilton-Berge 1-factorisation} (HB1F) is a 1-factorisation of a $k$-uniform hypergraph for which the union of each $k$-set of 1-factors has a Hamilton Berge cycle. In addition to showing some existence results of these generalisations, they also classified some known 1-factorisations as being C1Fs, HB1Fs, and U1Fs \cite{mitchell2022factorisation}. Of these known 1-factorisations, the infinite family of symmetric 1-factorisations of $K_{q+1}^3$ for $q\equiv 2\pmod 3$ found by Chen and Lu \cite{chen2017symmetric} is the subject of this paper. We denote this 1-factorisation by $\mathcal{F}_q,$ see \S \ref{ss:construction_explicit}. Davies et al.\ \cite{mitchell2022factorisation}, showed that for $q=2,5,8$ these 1-factorisations are C1Fs, U1Fs, UC1Fs, and HB1Fs, and that the 1-factorisations when $q=11,32$ are C1Fs and HB1Fs.

The main results of this paper determine when $\mathcal{F}_q$ is a C1F or a U1F.
\begin{theorem}\label{firstmain}
    $\mathcal{F}_q$ is a connected 1-factorisation if and only if $q\in \{2,5,11\}$ or $q=2^p$ for some odd prime $p$.
\end{theorem}
\begin{theorem}\label{secondmain}
    $\mathcal{F}_q$ is a uniform 1-factorisation if and only if $q\in \{2,5,8\}$, and in these cases it is a uniform-connected 1-factorisation.
\end{theorem}

\section{Preliminaries}\label{s:prelim}

Throughout this paper, $q$ denotes a prime power that satisfies $q\equiv 2 \pmod 3.$ Let $\mathbb{F}=\mathbb{F}_q$ be the finite field of order $q$, let $\mathbb{F}^*=\mathbb{F}\setminus \{0\}$, and let $V=\mathbb{F}\cup \{\infty\}$ be the vertex set of the complete 3-uniform hypergraph of order $q+1$, $K^{3}_{q+1}$. We will work with a particular 1-factorisation of $K^{3}_{q+1}$, originally given by Chen and Lu \cite{chen2017symmetric}. To describe this 1-factorisation, we first recall the action of the group $\PSL(2,q)$ on the projective line $\PG(1,q).$ 

\subsection{The action of $\mathbf{PSL(2,q)}$ on the projective line}

Recall that $V=\mathbb{F}\cup \{\infty\}$ can be identified with the projective line $\PG(1,q).$ (In homogeneous coordinates, identify $\infty$ with $[1:0]^T$ and $x\in \F$ with $[x:1]^T.$) The group $\PSL(2,q)$ is a quotient of $\SL(2,q):$
$$\PSL(2,q)=\left\lbrace \left(\begin{array}{cc}
            \alpha & \beta  \\
            \gamma & \delta
        \end{array}\right)\mid \alpha,\beta,\gamma,\delta \in \F;\ \alpha\delta-\beta\gamma =1 \right\rbrace\slash \sim ,$$
where two matrices are equivalent by $\sim $ if one is a scalar multiple of the other. The group $\PSL(2,q)$ acts faithfully on $\PG(1,q):$
$$\left(\begin{array}{cc}
            \alpha & \beta  \\
            \gamma & \delta
        \end{array}\right).\left[ \begin{array}{c}
            x \\
            y\end{array}\right]=
    \left[ \begin{array}{c}
            \alpha x+\beta y \\
            \gamma x+\delta y
        \end{array}\right].$$

It is often more convenient to think of this action in terms of fractional linear transformations on $\F\cup \{\infty \}.$ We may also write this as:
$$t_{\alpha,\beta,\gamma,\delta }(x):=\frac{\alpha x+\beta }{\gamma x+\delta }$$
for $x\in \F\cup \{\infty \},$ where $\frac{\alpha \infty+\beta}{\gamma \infty+\delta }=\frac{\alpha}{\gamma}$ for $\gamma \neq 0$, $\frac{\alpha \infty + \beta}{\delta} = \infty$ for $\alpha \neq 0$, and $\frac{\omega}{0}=\infty$ for $\omega \in \mathbb{F}^*$.

Recall that the action of $\PSL(2,q)$ on $\PG(1,q)$ is transitive. Note that we may think of $\PSL(2,q)$ as a subgroup of $\PGL(2,q)=\GL(2,q)/\sim ,$ and the action given above extends to one of $\PGL(2,q).$ Under this (extended) action, the stabiliser of the point $\infty$ is a subgroup. Indeed, it is the group of the invertible linear transformations of $\F\cup \{\infty\}:$
$$\PGL(2,q)_{\infty}=\{g_{\alpha,\beta}:=t_{\alpha,\beta,0,1}\mid \alpha \in \F^{\ast },\ \beta \in \F\}.$$

\subsection{The construction of Chen and Lu}\label{ss:construction_explicit}
The following construction was given by Chen and Lu \cite[Example 5.1, Lemma 5.2]{chen2017symmetric}. We present it here in a format convenient for our purposes. Let $f:V\rightarrow V$ be the map defined by
$$
    f(x)=\frac{1}{1-x} \text{  for $x\in V\setminus \{ 1,\infty \} $ and $f(1)=\infty$ and $f(\infty)=0$.}
$$
For any $\alpha \in \mathbb{F}^{\ast }$ and $\beta \in \mathbb{F}$ define $g_{\alpha,\beta}:$ $V\rightarrow V$ as $g_{\alpha,\beta}(x)=\alpha x + \beta$ for $x\in \F$ and $g_{\alpha,\beta}(\infty)=\infty$. Set $m_{\alpha,\beta}:=g_{\alpha,\beta} \circ f \circ g_{\alpha,\beta}^{-1}.$ Here $g_{\alpha,\beta}\in \PGL(2,q)$ and the conjugation makes sense in this group, so that $m_{\alpha,\beta}\in \PSL(2,q).$ Then
$$m_{\alpha,\beta}(x) = \beta + \frac{\alpha^2}{\alpha + \beta - x} \text{  for $x\in \F \setminus \{\alpha+\beta\}$ and $m_{\alpha,\beta}(\alpha+\beta)=\infty$, $m_{\alpha,\beta}(\infty)=\beta$.}$$
Note that $m_{1,0}=f$ and $m_{\alpha,\beta }^2=m_{\alpha,\beta }^{-1}.$ In particular, $m_{\alpha,\beta }\in \PSL(2,q)$ has order $3$ for any $\alpha \in \mathbb{F}^{\ast }$ and $\beta \in \mathbb{F}.$ It has no fixed points on $V=\PG(1,q).$ We also note that in the notation of \cite{chen2017symmetric}, $f=t_{0,1,-1,1}$ and $m_{\alpha,\beta }=t_{-\beta ,\alpha^2+\alpha\beta+\beta^2,-1,\alpha+\beta }$.

It follows that the sets
$$F_{\alpha,\beta}:=
    \left\lbrace \left\lbrace x,m_{\alpha,\beta}(x),m_{\alpha,\beta}^{-1}(x)\right\rbrace  \mid \, x\in \mathbb{F}\cup \{\infty\} \right\rbrace$$
are 1-factors of $K^3_{q+1}$ on $V.$ Set
$$\mathcal{F}_q=\{ F_{\alpha,\beta} \mid \alpha \in \mathbb{F}^{\ast }, \beta \in \mathbb{F}\}.$$
By \cite[Lemma 5.2]{chen2017symmetric} $\mathcal{F}_q$ is a 1-factorisation of $K_{q+1}^3$ for $q\geq 5$ if $q \equiv 2 \pmod 3$. In fact for $q=2$ it is the trivial 1-factorisation of $K_3^3$.  Recall also that $\mathcal{F}_q$ has $\frac{q(q-1)}{2}$ 1-factors. In particular, each is represented by exactly two $(\alpha,\beta)$ pairs: $F_{\alpha,\beta} = F_{-\alpha,\alpha+\beta}$. Using the notation of \cite{chen2017symmetric}, $F_{\eta^i,\beta }$ is $F_{i,\beta }$ in our notation, and their ${\mathcal{PG}}_{(q+1;3,\frac{q(q-1)}{2})}$ is our $\mathcal{F}_q.$

\begin{remark}\label{rmk:quadwithnosol}
    Observe that $f=m_{1,0}$ having no fixed points implies that no $x\in \mathbb{F}$ satisfies $x^2-x+1=0.$
\end{remark}

In the language of the action of $\PSL(2,q)$ on $\PG(1,q),$ each 1-factor $F_{\alpha,\beta }\in {\mathcal{F}}_q$ is the set of orbits of $\langle m_{\alpha,\beta }\rangle $ on $V=\PG(1,q).$ The elements $m_{\alpha,\beta }\in \PSL(2,q)$ are exactly the conjugates of $f=m_{1,0}$ in $\PGL(2,q)$ by the stabiliser $\PGL(2,q)_{\infty}.$ Each subgroup $\langle m_{\alpha,\beta }\rangle $ is isomorphic to the cyclic group of order $3$.

In this paper we are concerned with the subhypergraph spanned by two (or in the last section, three) 1-factors of ${\mathcal{F}}_q.$ The following remark says that by relabelling vertices of $V$ we may assume, without loss of generality, that one of the factors is $F_{1,0}.$

\begin{remark}\label{rmk:relabeling}
    Given two 1-factors $F_{\alpha_1,\beta _1},F_{\alpha_2,\beta _2}$ of ${\mathcal{F}}_q,$ there is a $1$-factor $F_{\alpha_0,\beta _0}$ so that the subhypergraph $H$ spanned by the pair $(F_{\alpha_1,\beta _1},F_{\alpha_2,\beta _2})$ is isomorphic to the one, $H'$, spanned by the pair $(F_{1,0},F_{\alpha_0,\beta _0}).$ Indeed, let $(m_{\alpha_1,\beta _1},m_{\alpha_2,\beta _2})$ be a pair of elements of $\PSL(2,q).$ Consider $g=g_{\alpha_1^{-1},-\alpha_1^{-1}\beta_1},$ and set $\alpha_0=\alpha_1^{-1}\alpha_2,$ $\beta_0=\alpha_1^{-1}(\beta_2-\beta_1).$ Then $(gm_{\alpha_1,\beta _1}g^{-1},gm_{\alpha_2,\beta _2}g^{-1})=(f,m_{\alpha_0,\beta_0})=(m_{1,0},m_{\alpha_0,\beta_0}).$ Therefore $g$ gives an isomorphism between the hypergraphs $H$ and $H'.$
\end{remark}

\section{Connected 1-Factorisations of $\mathcal{F}_q$}\label{s:C1Fs}
In this section we prove Theorem 1.1, noting that $\mathcal{F}_2,\mathcal{F}_5,\mathcal{F}_8$, and $\mathcal{F}_{11}$ were shown to be C1Fs in \cite{mitchell2022factorisation}. The graph theoretic properties of the factorisation $\mathcal{F}_q$ can be rephrased in terms of the action of $\PSL(2,q)$ on $\PG(1,q).$ We explain this first. Note that whenever we refer to the action of a group on $\PG(1,q),$ we understand this to mean that the group is embedded into $\PSL(2,q)$ or $\PGL(2,q),$ and the action is the one described in \S\ref{s:prelim}.

\begin{lemma}\label{lem:C1F_iff_gensgstransitive}
    The 1-factorisation $\mathcal{F}_q$ is a connected 1-factorisation if and only if the subgroup $H_{\alpha, \beta }=\langle f,m_{\alpha,\beta}\rangle $  acts transitively on $\PG(1,q)$ for each $(\alpha,\beta)\in \mathbb{F}_q^*\times \mathbb{F}_q.$ Such a subgroup $H_{\alpha, \beta }$ has at least four elements of order $3.$
\end{lemma}
\begin{proof}
    As seen above, two vertices of $V$ are in the same edge of $F_{\alpha,\beta}$ if and only if they are in the same orbit under the action of $\langle m_{\alpha,\beta }\rangle$ on $\PG(1,q).$ Similarly, the union of two 1-factors $F_{1,0}\cup F_{\alpha,\beta}$ is connected if and only if the action of the subgroup $H:=\langle f,m_{\alpha,\beta}\rangle $ acts transitively on $\PG(1,q).$ Indeed, for $x,y\in V$ a path from $x$ to $y$ along edges from $F_{1,0}$ and $F_{\alpha,\beta}$ corresponds to an element $w\in \PSL(2,q)$ such that $w(x)=y$ and $w=f^{i_0} \circ m_{\alpha,\beta}^{i_1}\circ f^{i_2}\circ \cdots \circ m_{\alpha,\beta}^{i_\ell}\in H$, $i_k\in\{-1,0, 1\}$ ($0\leq k\leq \ell$). The first statement now follows from Remark \ref{rmk:relabeling}.

    Furthermore if $F_{\alpha,\beta}\neq F_{1,0}$ then $f,f^{-1},m_{\alpha,\beta },m_{\alpha,\beta }^{-1}$ are four distinct elements of $H_{\alpha, \beta }$ that each have order 3.
\end{proof}

According to Lemma \ref{lem:C1F_iff_gensgstransitive}, the subgroups of $\PSL(2,q)$ are relevant for deciding whether $\cF_q$ is a C1F or not. We recall the following theorem classifying all subgroups of ${{PSL}}(2,q)$.

\begin{theorem} \cite[Theorem 6.25]{Suzuki1982GroupTheory} \label{subgroups}
    Let $q=p^s,$ $d=\gcd(2,p-1).$ Every subgroup of $PSL(2,q)$ is isomorphic to (at least) one of the following.
    \begin{enumerate}[(a)]
        \item The dihedral groups of orders $\frac{2(q\pm 1)}{d}$ and their subgroups.
        \item A group $K$ of order $\frac{q(q-1)}{d}$ and its subgroups. A Sylow $p$-subgroup $Q$ of $K$ is elementary Abelian, normal in $K,$ and the factor group $K/Q$ is a cyclic group of order $\frac{q-1}{d}.$
        \item $A_4,S_4,$ or $A_5.$
        \item ${{PSL}}(2,p^r)$ or ${{PGL}}(2,p^r)$ where $r|s.$ Note ${{PSL}}(2,2^r)={{PGL}}(2,2^r)$. Further, for $p>2$, ${{PGL}}(2,p^r)$ does not occur if $r=s$.
    \end{enumerate}
\end{theorem}

We wish to understand which of the above subgroups are isomorphic to the subgroups obtained by taking the union of a pair of 1-factors of $\mathcal{F}_q$. We may eliminate some right away. This is the content of the following.

\begin{corollary}\label{cor:whatsubgroups}
    Let $q$ be an odd prime such that $q\equiv 2\pmod 3$. Let $F_{\alpha,\beta}\in {\mathcal{F}}_q$ be a 1-factor that is different from $F_{1,0},$ and let $H_{\alpha, \beta }=\langle f,m_{\alpha,\beta}\rangle .$ Then $H_{\alpha,\beta }$ is isomorphic to $A_4,$ $S_4$ or $A_5,$ or $H_{\alpha,\beta }=\PSL(2,q).$
\end{corollary}
\begin{proof}
    In the notation of Theorem \ref{subgroups} we have $s=1$ and $d=2.$ Further, $q\equiv 2\pmod 3$ implies that $3\nmid \frac{q(q-1)}{2}.$ Since by Lemma \ref{lem:C1F_iff_gensgstransitive} $H_{\alpha,\beta }$ has at least four distinct elements of order $3,$ it follows that $H_{\alpha,\beta }$ is neither dihedral, nor a subgroup of a group $K$ of order $\frac{q(q-1)}{d}.$ Therefore it is either the entire group $\PSL(2,q),$ or isomorphic to one of $A_4,$ $S_4$ or $A_5.$
\end{proof}

We now show that for certain conditions on $q$ we can find pairs of 1-factors whose corresponding subgroup does not act transitively on $PG(1,q)$.

\begin{lemma}\label{c1f:primepowercomb}
    Let $q\equiv 2 \pmod 3$. If $q=p^\ell$ for some odd prime $p\geq 5$ and some integer $\ell\geq 2$, or $q=2^r$ for some odd composite $r$, then $\mathcal{F}_q$ is not a C1F.
\end{lemma}

\begin{proof}
    Under the conditions of the lemma we have a $5\leq q'<q$ such that $\F_{q'}\subset \F.$ Indeed if $q=p^{\ell}$, then take $q'=p.$ If $q=2^r,$ then take $q=2^{r'}$ for $r'$ the smallest prime divisor of $r.$ (Here $q'\equiv 2\pmod 3$ from the same property of $q.$)

    Consider some $\alpha \in \mathbb{F}_{q'}^\ast$ and $\beta \in \mathbb{F}_{q'}$ such that $F_{\alpha,\beta }\neq F_{1,0}.$ Such a pair exists, since there are at least $5(5-1)-2=20$ choices for the pair $(\alpha,\beta)$ outside of $\{(1,0),(1,-1)\}.$ For such a choice, all vertices in $\mathbb{F}_{q'}\cup \{\infty\}$ will appear in edges only with other vertices of $\mathbb{F}_{q'}\cup \{\infty\}$ in both $F_{\alpha,\beta}$ and $F_{1,0}$. Thus the subgraph $F_{\alpha,\beta }\cup F_{1,0}$ has no edges that include vertices from both $\F_{q'}\cup \{\infty\}$ and $\F\setminus \F_{q'},$ and is therefore disconnected.

    In other words, the choice of $(\alpha,\beta )$ implies that $H_{\alpha,\beta }=\langle f,m_{\alpha,\beta }\rangle $ is a subgroup of $\PSL(2,q').$ Therefore it leaves $\PG(1,q')\subset \PG(1,q)$ invariant, and does not act transitively on $\PG(1,q).$
\end{proof}

Therefore $\cF_q$ may only be a C1F if $q$ is an odd prime, or $q=2^p$ for some odd prime $p$.

\begin{lemma} \label{c1f:oddprimec1f}
    Let $q$ be an odd prime such that $q\equiv 2\pmod 3$. If $q>11$ then $\mathcal{F}_q$ is not a C1F.
\end{lemma}
\begin{proof}
    It follows from Lemma \ref{lem:C1F_iff_gensgstransitive} that it suffices to show that there are two different 1-factors, $F_{\alpha,\beta }$ and $F_{\alpha',\beta'}$ such that the action of the subgroup $H=\langle m_{\alpha,\beta }, m_{\alpha',\beta '}\rangle $ of $\PSL(2,q)$ is not transitive on $\PG(1,q).$ This follows if we show that there are two elements $m_{\alpha,\beta }$ and $m_{\alpha',\beta '},$ corresponding to different 1-factors, that are contained in a subgroup of $\PSL(2,q)$ that is isomorphic to $A_4.$ Indeed, this would mean that $|H|\leq |A_4|=12,$ and in turn that $H$ is not transitive on $\PG(1,q),$ which has more than $12$ elements.

    From \cite{cameron_3_cycles_psl} we know that if $q$ is odd and $3|q+1$ then $\PSL(2,q)$ contains $\frac{q(q^2-1)}{24}$ copies of $A_4$ and that each $C_3$ subgroup of ${{PSL}}(2,q)$ is contained in $\frac{q+1}{3}$ subgroups $A_4.$ There are $\frac{q(q-1)}{2}$ 1-factors in $\cF_q$, each corresponding to a distinct copy of $C_3$ inside $\PSL(2,q).$ Each one of these elements is contained in $\frac{q+1}{3}$ copies of $A_4$. If no two were contained in the same copy of $A_4,$ then there would be at least $\frac{q(q-1)(q+1)}{6}$ copies of $A_4$ in a subgroup of $PSL(2,q).$ Since there are only $\frac{q(q^2-1)}{24}$ such copies of $A_4$, this is a contradiction. This completes the proof.
\end{proof}

Thus for $\mathcal{F}_q$ to be a C1F, $q<11$ or $q=2^p$ for some odd prime $p$. To show that $\mathcal{F}_{2^p}$ is indeed a C1F, we need to make use of the following facts.
\begin{lemma} \cite[\S 2.5 Remarque]{Serre1972}\label{a4s4a5}
    The group $PSL(2,2^p)={{PGL}}(2,2^p)$ has no subgroup isomorphic to $S_4.$ It contains a subgroup isomorphic to $A_4$ or $A_5$ under the following conditions:
    \begin{enumerate}[(a)]
        \item $PSL(2,2^p)$ contains $A_4$ if and only if there exists an $x\in \mathbb{F}_{2^p}$ such that $x^2+x=1$.
        \item $PSL(2,2^p)$ contains $A_5$ if and only if there exist $x,y,z\in \mathbb{F}_{2^p}$ such that $x^2+x=1$ and $y^2+z^2=-1$.
    \end{enumerate}
\end{lemma}

\begin{lemma}\label{c1f:2primepowerc1f}
    If $q=2^p$ for some odd prime $p$, then $\mathcal{F}_q$ is a C1F.
\end{lemma}
\begin{proof}
    By Lemma \ref{lem:C1F_iff_gensgstransitive} it suffices to show that $H_{\alpha,\beta}=\langle f,m_{\alpha,\beta} \rangle$ acts transitively on $\PG(1,q)$ whenever $F_{\alpha,\beta }\neq F_{1,0}.$ By Remark \ref{rmk:quadwithnosol} there is no $x\in \F_q$ such that $x^2+x+1=0.$ In characteristic $2$ this is equivalent to $x^2+x=1.$ Thus by Lemma \ref{a4s4a5} above, $H_{\alpha,\beta }$ is not isomorphic to $S_4,$ $A_4$ or $A_5.$ It follows from Corollary \ref{cor:whatsubgroups} that $H_{\alpha,\beta }=\PSL(2,q)$ is the entire group. It therefore acts transitively on $\PG(1,q).$ This completes the proof.
\end{proof}

The proof of Theorem \ref{firstmain} follows by combining Lemmas \ref{c1f:primepowercomb}, \ref{c1f:oddprimec1f}, and \ref{c1f:2primepowerc1f} with the knowledge from \cite{mitchell2022factorisation} that $\mathcal{F}_2,\mathcal{F}_5, $ and $\mathcal{F}_{11}$ are C1Fs.

\section{Uniform 1-Factorisations} \label{s:U1Fs}
In this section we prove Theorem 1.2, noting that $\mathcal{F}_5$ and $\mathcal{F}_8$ were shown to be U1Fs (and UC1Fs) in \cite{mitchell2022factorisation}.

For two distinct 1-factors $F_1$ and $F_2$ of a hypergraph, we say that a pair of vertices, $B=\{v_1,v_2\}$, is \emph{repeated} in the pair $F_1$ and $F_2$ if $B\subseteq e$ for some edge $e\in F_1$ and $B\subseteq e'$ for some edge $e'\in F_2$. We call the number of repeated pairs in a pair of 1-factors the \emph{pair overlap number}. If each pair of distinct 1-factors of a 1-factorisation have the same pair overlap number, we call that the \emph{pair overlap number} of the 1-factorisation. Davies et al.\ \cite{mitchell2022factorisation} showed that if a U1F of $K_n^3$ exists then the pair overlap number of the 1-factorisation is 2. Thus in order to prove that $\mathcal{F}_q$ is not a U1F, we need only show that there exist two distinct 1-factors with pair overlap number not equal to 2.

Let $F_{1,0}$ and $F_{\alpha,\beta}$ be distinct 1-factors of $\mathcal{F}_q$, with corresponding functions $f$ and $m_{\alpha,\beta}$ for $\alpha \in \mathbb{F}^*$ and $\beta \in \mathbb{F}$. Recall that this is the case if and only if $(\alpha,\beta)\not\in \{(1,0),(-1,1)\}$ in $\mathbb{F}$. Observe that the pair overlap number of $F_{1,0}$ and $F_{\alpha,\beta}$ is
$$
    |\{x\in \mathbb{F} \cup \{\infty\} \,\, : \,\, f(x) = m_{\alpha,\beta}(x)\}| + |\{x\in \mathbb{F} \cup \{\infty\} \, : \, f^{-1}(x) = m_{\alpha,\beta}(x)\}|.
$$

This means that every repeated pair corresponds to a solution to either $f(x)=m_{\alpha,\beta}(x)$ or $f^{-1}(x)=m_{\alpha,\beta}(x)$.

We will now consider the number of solutions for $f(x)=m_{\alpha,\beta}(x)$ with values of $\alpha$ and $\beta$ that result in $F_{\alpha,\beta}$ being distinct from $F_{1,0}$. A solution to $f(x)=m_{\alpha,\beta}(x)$ gives us the equation
$$
    \frac{1}{1-x}=\beta + \frac{\alpha^2}{\alpha+\beta-x}.$$
We note that for $x=\infty$, $f(\infty)=0$ and $m_{\alpha,\beta}(\infty)=\beta$ so there is at least one solution if $\beta = 0$, and only one if $\beta=0$, $\alpha = -1$, and $F_{1,0}\neq F_{-1,0}$. Further, if $\beta=0$, $\alpha\not\in \{-1,1\}$, then we also get the solution $x=\frac{\alpha}{1+\alpha}$. If $\alpha+\beta = 1$ then the only solutions are $x=1$ and $x=-\alpha$. (Note $\left(\alpha,\beta\right)\neq \left(1,0\right)$. We may have $\alpha=-1$.) We now consider the case where $\alpha+\beta\neq 1$. Then $f(x)=m_{\alpha,\beta}(x)$ implies $x\not \in \{1,\alpha+\beta\}$, every solution is in $\mathbb{F}$ and $f(x)=m_{\alpha,\beta}(x)$ is equivalent to
\begin{align}\label{f=m}
    0=(\alpha^2+\alpha\beta-\alpha +\beta^2-\beta)-x(\alpha^2+\alpha\beta+\beta^2+\beta-1)+\beta x^2.
\end{align}

We will now consider the number of solutions for $f^{-1}(x)=m_{\alpha,\beta}(x)$ for values of $\alpha$ and $\beta$ that result in $F_{\alpha,\beta}$ being distinct from $F_{1,0}$. A solution to $f^{-1}(x)=m_{\alpha,\beta}(x)$ gives us the equation
$$
    1-\frac{1}{x}=\beta + \frac{\alpha^2}{\alpha+\beta-x}.$$
Note that $f^{-1}(\infty)=1$ and $m_{\alpha,\beta}(\infty)=\beta$, so there is at least one solution if $\beta = 1$. If $\beta=1$ we may assume $\alpha \neq -1$ as we require $F_{\alpha,\beta}$ to be distinct from $F_{1,0}$, and then we have the solution $x=\frac{1}{1-\alpha}$ which is $\infty$ if $\alpha=1$. If $\alpha+\beta =0$ then $f^{-1}(0)=m_{\alpha,\beta}(0)=\infty,$ so $x=0$ and $x=1-\alpha$ give solutions. Now assume that $\beta \neq 1$ and $\alpha+\beta \neq 0.$ Then $f^{-1}(x)=m_{\alpha,\beta}(x)$ implies $x\notin \{0,\alpha+\beta \},$ every solution is in $\F$ and $f^{-1}(x)=m_{\alpha,\beta}(x)$ is equivalent to
\begin{align}\label{finv=m}
    0=(1-\beta)x^2+x(\alpha^2+\alpha\beta +\beta^2-\alpha-\beta-1)+(\alpha+\beta).
\end{align}

To summarise the above we have the following.

$$\begin{array}{p{3.4cm}|p{11cm}}
        $\hskip .19cm\text{Conditions}$                                                                                                             & $\{x\in \F\cup\{\infty\}\mid f(x)=m_{\alpha,\beta}(x) \}$                                                          \\ \hline \hline
        $\hskip .19cm \beta=0,\  \alpha = -1$                                                                                                       & $\{\infty\}$                                                                                                       \\ \hline
        $\hskip .19cm \beta=0,\  \alpha\not\in \{-1,1\}$                                                                                            & $\{\infty, \frac{\alpha}{1+\alpha }\}$                                                                             \\ \hline
        $\hskip .19cm \beta \neq 0,\ \alpha+\beta=1$                                                                                                & $\{1,-\alpha \}$                                                                                                   \\ \hline
        $\begin{array}{p{3.5cm}}$\beta \neq 0,\ \alpha+\beta \neq 1 ,$\\ $((\alpha,\beta)\neq (-1, 1))$\end{array}$ & $\{x\in \F\mid \beta x^2-(\alpha^2+\alpha\beta+\beta^2+\beta-1)x+(\alpha^2+\alpha\beta+\beta^2-\alpha-\beta)=0 \}$
    \end{array}$$

$$\begin{array}{p{3.4cm}|p{11cm}}
        $\hskip .19cm\text{Conditions}$                                                                                                            & $\{x\in \F\cup\{\infty\}\mid f^{-1}(x)=m_{\alpha,\beta}(x) \}$                                    \\ \hline \hline
        $\hskip .19cm \beta=1,\  \alpha = 1$                                                                                                       & $\{\infty\}$                                                                                      \\ \hline
        $\hskip .19cm\beta=1,\ \alpha\not\in \{-1,1\}$                                                                                             & $\{\infty, \frac{1}{1-\alpha }\}$                                                                 \\ \hline
        $\hskip .19cm\beta \neq 1,\ \alpha+\beta=0$                                                                                                & $\{0,1-\alpha \} $                                                                                \\ \hline
        $\begin{array}{p{3.5cm}}$\beta \neq 1,\ \alpha+\beta \neq 0,$\\$\left((\alpha,\beta)\neq (-1, 1)\right)$\end{array}$ & $\{x\in \F\mid (1-\beta) x^2+(\alpha^2+\alpha\beta+\beta^2-\alpha -\beta-1)x+(\alpha+\beta)=0 \}$
    \end{array}$$

We use the information in the tables to show that $\mathcal{F}_q$ is not a U1F if $q\notin \{5,8\}.$ The cases of $5|q$ and $2|q$ are treated separately from that of other primes. We start with the case of primes greater than 5.

\begin{lemma}\label{lem:oddnot5powernotU1F}
    Let $q=p^{\ell}$ for some prime $p> 5$ and some integer $\ell \geq 1$ such that $q\equiv 2\pmod 3$. Then $\mathcal{F}_q$ is not a U1F.
\end{lemma}

\begin{proof}
    Let $F_{1,0}$ and $F_{-1,0}$ be 1-factors of $\mathcal{F}_q$; we shall prove that the pair overlap number of this pair of 1-factors is not 2. $F_{1,0}$ and $F_{-1,0}$ are distinct, and from above we know that there is only one repeated pair corresponding to a solution to $f(x)=m_{-1,0}(x)$. Further, we know that  $\{x\in \F_q \cup \{\infty\}\  |\  f^{-1}(x)=m_{-1,0}(x)\}=\{x\in \F_q \ |\  x^2+x-1 = 0\}$, and $x^2+x-1 = 0$ will have 2 solutions in $\F_q$ if $5$ is a quadratic residue, and 0 if not. Thus the pair overlap number of this pair of 1-factors must be either 1 or 3, and thus $\mathcal{F}_q$ is not a U1F.
\end{proof}

\begin{lemma}\label{lem:powersof5notU1F}
    Let $q=5^{\ell}$ for some integer $\ell>1$ such that $q\equiv 2\pmod 3$. Then $\mathcal{F}_q$ is not a U1F.
\end{lemma}
\begin{proof}
    As in the proof of Lemma \ref{lem:oddnot5powernotU1F} we show that there is a choice of $\alpha,\beta$ such that the pair overlap number of $F_{1,0}$ and $F_{\alpha,\beta }$ is not 2. This implies that $\mathcal{F}_q$ is not a U1F. We shall show that for $\alpha\in \F\setminus\F_5$ the factor $F_{1,0}$ has a pair overlap number of $4$ with at least one of $F_{\alpha,-\alpha},$ $F_{\alpha,1-\alpha}$ or $F_{\alpha^2,1-\alpha^2}.$

    It follows from the tables above that if $\alpha\in \F\setminus\F_5$ and we set $\beta =-\alpha$ then $f^{-1}(x)=m_{\alpha,-\alpha }$ has the two distinct solutions $0$ and $1-\alpha$. The solutions of $f(x)=m_{\alpha,-\alpha }$ are $x\in \F$ such that $\alpha x^2+(\alpha^2-\alpha-1)x-\alpha^2=0.$ The discriminant is $D_1=(\alpha-1)^2\cdot (\alpha^2-\alpha+1).$ Therefore the pair overlap number between $F_{1,0}$ and $F_{\alpha,-\alpha}$ is $4$ if $\alpha^2-\alpha+1$ is a square in $\mathbb{F}$.

    Now set $\beta =1-\alpha$ in the tables above. If $\alpha\in \F\setminus\F_5$ then $f(x)=m_{\alpha,1-\alpha }$ has the two distinct solutions, $1$ and $-\alpha$. The solutions of $f^{-1}(x)=m_{\alpha,1-\alpha}$ are $x\in \F$ such that $\alpha x^2+(\alpha^2 -\alpha-1)x+1=0.$ The discriminant is $D_2=(\alpha+1)^2\cdot (\alpha^2+\alpha+1).$  Therefore the pair overlap number between $F_{1,0}$ and $F_{\alpha,1-\alpha}$ is $4$ if $\alpha^2+\alpha+1$ is a square in $\mathbb{F}$.

    Now take an $\alpha	\in \F\setminus\F_5.$ If $\alpha^2-\alpha+1\in \F^2$ or $\alpha^2+\alpha+1\in \F^2$ then the pair overlap number of $F_{1,0}$ with $F_{\alpha,-\alpha}$ or with $F_{\alpha,1-\alpha}$ is not $2$ by the above paragraphs. Recall that $\F^{\ast}$ is a cyclic group, therefore the product of two non-squares is a square. Therefore if $\alpha^2-\alpha+1\notin \F^2$ and $\alpha^2+\alpha+1\notin \F^2,$ then their product is a square:
    $(\alpha^2-\alpha+1)(\alpha^2+\alpha+1)=(\alpha^2)^2+\alpha^2+1\in \F^2.$

    Observe that $q\equiv 2\pmod 3$ implies that $\ell$ is odd. Therefore $\F$ does not contain the field of $25$ elements. This implies that for $\alpha\in \F\setminus\F_5$ we have $\alpha^2\in \F\setminus\F_5$. Thus using similar working to above, the pair overlap number of $F_{1,0}$ and $F_{\alpha^2,1-\alpha^2}$ is $4$.
\end{proof}

We now turn our attention to the case where $q= 2^{\ell}$ for $\ell$ an odd integer, $\ell>3$. We shall show that then $\mathcal{F}_q$ is not a U1F by proving that there is an $\alpha\in \F_q\setminus \{0,1\}$ such that the pair overlap number of $F_{1,0}$ with $F_{\alpha,1}$ or $F_{\alpha,0}$ is not $2$. As in the case of odd characteristic, the proof involves considering the number of solutions of the quadratic equations \eqref{f=m} and \eqref{finv=m} in special cases. To do so we recall the following.

\begin{lemma}\label{lem:Galoisfacts}
    Let $\ell$ be a positive integer and set $\F=\F_{2^{\ell}}.$ The field extension $\F|\F_2$ is cyclic, its Galois group generated by the Frobenius automorphism $x\mapsto x^2.$ The trace map: ${\mathrm{Tr}}={\mathrm{Tr}}_{\F_2}^{\F}:\F\rightarrow \F_2$ given by
    \begin{equation}\label{eq:tracedef}
        {\mathrm{Tr}}(x)={\mathrm{Tr}}_{\F_2}^{\F_{2^{\ell}}}(x)=\sum_{i=0}^{\ell-1}x^{2^i}
    \end{equation}
    is an $\F_2-$linear map. For any $x\in \F$ we have $x^{2^{\ell}}=x.$ For $x\in \F$ there exists an $r\in \F$ such that $x=r^2+r$ if and only if ${\mathrm{Tr}}(x)=0.$
    A quadratic equation $x^2+Lx+C$ with $L\neq 0$ has two solutions in $\F$ if ${\mathrm{Tr}}\left(\frac{C}{L^2}\right)=0,$ and zero solutions otherwise. If $\ell$ is an odd integer then ${\mathrm{Tr}}(1)=1.$
\end{lemma}
\begin{proof}
    For general facts about the trace map, see \cite[\S 5]{LangAlgebra}. The fact that $x=r^2+r$ has a solution if and only if ${\mathrm{Tr}}(x)=0$ is the additive form of Hilbert's Theorem 90 \cite[Theorem 6.3]{LangAlgebra}. The statement about the number of roots of a quadratic equation follows from the Artin-Schreier theorem \cite[Theorem 6.4]{LangAlgebra} by a change of variables. See for example \cite[Proposition 1]{pommerening2000quadratic}.
\end{proof}

We now have the following.

\begin{lemma}\label{lem:ifinR_4overlap-oddpowerof2}
    Let $q=2^\ell$ with $\ell$ a positive odd integer, and set $\F=\F_{2^{\ell}},$ ${\mathrm{Tr}}={\mathrm{Tr}}^{\F}_{\F_2}$ as in \eqref{eq:tracedef}. If there exists an $\alpha \in \F^{\ast }$ such that
    \begin{equation}\label{eq:alphaaltcond}
        {\mathrm{Tr}}\left(\frac{\alpha}{(\alpha^2+\alpha+1)^2}\right)=0\text{ or }{\mathrm{Tr}}\left(\frac{\alpha^2}{(\alpha^2+\alpha+1)^2}\right)=0
    \end{equation}
    then ${\mathcal{F}}_q$ is not a U1F.
\end{lemma}
\begin{proof}
    First note that for $\alpha=1$ the conditions in \eqref{eq:alphaaltcond} are not satisfied since $\mathrm{Tr}(1)\neq 0$, so we only consider $\alpha\neq 1$. Therefore $F_{1,0}$ is different from both $F_{\alpha,0}$ and from $F_{\alpha,1}.$ Recall the discussion at the beginning of the section about the pair overlap number of the two 1-factors, $F_{1,0}$ and $F_{\alpha,\beta }.$ Specialising the corresponding tables to characteristic $2$ and $\beta\in \{0,1\}$ we find that the solutions are as follows.
    $$\begin{array}{c|c|c}
            \text{Conditions}       & \{x\in \F\cup\{\infty\}\mid f(x)=m_{\alpha,\beta}(x) \} & \{x\in \F\cup\{\infty\}\mid f^{-1}(x)=m_{\alpha,\beta}(x) \} \\ \hline \hline
            \beta=0\ (\alpha\neq 1) & \{\infty, \frac{\alpha}{1+\alpha }\}                    & \{x\in \F\mid x^2+(\alpha^2+\alpha +1)x+\alpha=0 \}          \\ \hline
            \beta=1\ (\alpha\neq 1) & \{x\in \F\mid x^2+(\alpha^2+\alpha+1)x+\alpha^2=0 \}    & \{\infty, \frac{1}{1+\alpha }\}
        \end{array}$$
    It follows that if ${\mathrm{Tr}}\left(\frac{\alpha}{(\alpha^2+\alpha+1)^2}\right)=0$ then the quadratic equation
    $x^2+(\alpha^2+\alpha+1)x+\alpha=0$ has two solutions by Lemma \ref{lem:Galoisfacts} and then $F_{1,0}$ and $F_{\alpha,0}$ have a pair overlap number of $4$. On the other hand if ${\mathrm{Tr}}\left(\frac{\alpha^2}{(\alpha^2+\alpha+1)^2}\right)=0$ then $x^2+(\alpha^2+\alpha+1)x+\alpha^2=0$ has two solutions and therefore $F_{1,0}$ and $F_{\alpha,1}$ have a pair overlap number of $4$. Thus the pair overlap number of $\mathcal{F}_q$ is not two and this implies that $\mathcal{F}_q$ is not a U1F.
\end{proof}

We shall show that the conditions of Lemma \ref{lem:ifinR_4overlap-oddpowerof2} are always met by an $\alpha\in \F^{\ast }$ when $\F=\F_{2^t}$ for $t>3.$ We do this in two steps.

\begin{lemma}\label{lem:traceofeltplusrecipnonzero}
    Let $\ell$ be a positive odd integer and set $\F=\F_{2^{\ell}},$ ${\mathrm{Tr}}={\mathrm{Tr}}^{\F}_{\F_2}$ as in \eqref{eq:tracedef}. If there is no $\alpha\in \F^{\ast }$ satisfying \eqref{eq:alphaaltcond} then for every $x\in \F\setminus\{0,1\}$ we have ${\mathrm{Tr}}\left(x+\frac{1}{x}\right)=1.$
\end{lemma}
\begin{proof}
    Suppose that there is no $\alpha \in \mathbb{F}^*$ satisfying \eqref{eq:alphaaltcond}. Recall that since ${\mathrm{Tr}}:\F\rightarrow \F_2$ is a surjective $\F_2-$linear map, its kernel is an index $2$ subgroup of the additive group $\F$, call this subgroup $R$. Every element not in $R$ has trace $1.$ It therefore suffices to show that $x$ and $\frac{1}{x}$ are in different cosets of $R$ for every $x\in \F\setminus\{0,1\}.$

    First assume that $x\neq 1$ and ${\mathrm{Tr}}(x)=1$. Since ${\mathrm{Tr}}(1)=1$ we have ${\mathrm{Tr}}(x+1)=0$ and so $x=r^2+r+1$ for some $r\in \F\setminus \{0,1\}.$ Since neither $r$ nor $r+1$ satisfy the conditions in \eqref{eq:alphaaltcond} of Lemma \ref{lem:ifinR_4overlap-oddpowerof2}, we have that
    $${\mathrm{Tr}}\left(\frac{r^2}{(r^2+r+1)^2}\right)=1\text{ and }{\mathrm{Tr}}\left(\frac{r+1}{(r^2+r+1)^2}\right)=1.$$
    It follows that the sum of these elements has trace zero:
    $$0={\mathrm{Tr}}\left(\frac{r^2}{(r^2+r+1)^2}+\frac{r+1}{(r^2+r+1)^2}\right)={\mathrm{Tr}}\left(\frac{1}{r^2+r+1}\right)={\mathrm{Tr}}\left(\frac{1}{x}\right).$$

    Thus taking reciprocals maps $\{h\in \F\setminus \{1\}\, : \,{\mathrm{Tr}}(h)=1\}$ into $\{h\in \F^{\ast }\, :\, {\mathrm{Tr}}(h)=0\}.$ These sets both have $2^{\ell-1}-1$ elements, therefore this is in fact a bijection. That is, for any $x\in \F\setminus\{0,1\}$ we have ${\mathrm{Tr}}(x)\neq {\mathrm{Tr}}\left(\frac{1}{x}\right),$ hence ${\mathrm{Tr}}\left(x+\frac{1}{x}\right)=1.$
\end{proof}

Putting Lemmas \ref{lem:ifinR_4overlap-oddpowerof2} and \ref{lem:traceofeltplusrecipnonzero} together we see that ${\mathcal{F}}_q$ is not a U1F for $q=2^{\ell},$ $\ell$ an odd integer greater than 3, unless a somewhat strange condition is satisfied by $\F=\F_{2^{\ell}}$. We show that this condition is indeed strange, and is in fact not satisfied for any $\ell>3.$

\begin{lemma}\label{lem:nosuchthingasalwaystrace1}
    Let $\ell$ be a positive odd integer and set $\F=\F_{2^{\ell}},$ ${\mathrm{Tr}}={\mathrm{Tr}}^{\F}_{\F_2}$ as in \eqref{eq:tracedef}. If $\ell>3$ then not all $x\in \F\setminus\{0,1\}$ satisfy
    \begin{equation}\label{eq:trace1withrecip}
        {\mathrm{Tr}}\left(x+\frac{1}{x}\right)=1.
    \end{equation}
\end{lemma}
\begin{proof}
    Observe that if $x$ satisfies \eqref{eq:trace1withrecip} then we have:
    \begin{equation}
        \begin{split}
            x^{2^{\ell-1}}\cdot {\mathrm{Tr}}\left(x+\frac{1}{x}\right)=x^{2^{\ell-1}}\cdot \sum_{i=0}^{\ell-1}\left(x^{2^i}+x^{-2^i}\right)&=  x^{2^{\ell-1}}\\
            \sum_{i=0}^{\ell-1}\left(x^{2^{\ell-1}+2^i}+x^{2^{\ell-1}-2^i}\right)&=  x^{2^{\ell-1}}\\
            x+\sum_{i=0}^{\ell-2} \left(x^{2^{\ell-1}+2^i}\right)+x^{2^{\ell-1}}+\sum_{i=0}^{\ell-1}\left(x^{2^{\ell-1}-2^i}\right)&=  0
        \end{split}
    \end{equation}
    where the last line follows from the fact that $x^{2^{\ell}}=x$ for every $x\in \F$ (see Lemma \ref{lem:Galoisfacts}). The left-hand side here is a polynomial of degree $2^{l-1}+2^{l-2}$ (with coefficients in $\F_2$). Therefore it has at most $2^{l-1}+2^{l-2}$ roots in the field $\F$. If $\ell>3$ then $2^{\ell-1}+2^{\ell-2}<2^{\ell}-2.$ Therefore there is an $x\in \F\setminus\{0,1\}$ that does not satisfy \eqref{eq:trace1withrecip}.
\end{proof}%

Combining these gives us the following.

\begin{corollary}\label{cor:powerof2notU1F}
    Let $q=2^\ell$ for some odd integer $\ell>3$. Then $\mathcal{F}_q$ is not a U1F.
\end{corollary}

The proof of Theorem \ref{secondmain} then follows from Lemmas \ref{lem:oddnot5powernotU1F}, \ref{lem:powersof5notU1F}, along with Corollary \ref{cor:powerof2notU1F}, and the knowledge from \cite{mitchell2022factorisation} that $\mathcal{F}_2,\mathcal{F}_5,\mathcal{F}_8$ are both U1Fs and C1Fs.

\section{Hamilton-Berge 1-Factorisations}

A necessary condition for a 1-factorisation of $K_n^k$ to be a Hamilton-Berge 1-factorisation is that the union of each $k$-set of 1-factors is connected. We remark that in the proof of Lemma \ref{c1f:primepowercomb}, the subfield is actually large enough to allow us to find three 1-factors whose union is disconnected. We also note that in the proof of Lemma \ref{c1f:oddprimec1f} there are not enough copies of $A_4$ to ensure that no three distinct copies of $C_3$ are contained in the same copy of $A_4$. From these two remarks it follows that if $q=p^\ell$ for some prime $p\geq 5$ and some integer $\ell\geq 2$, or if $q>11$ is prime,  or if $q=2^r$ for some odd composite $r$, then $\mathcal{F}_q$ cannot be an HB1F. Finally, it follows from Lemma \ref{c1f:2primepowerc1f} that if $q=2^p$ for some odd prime $p$, then $\mathcal{F}_q$ satisfies the property that the union of each set of three distinct 1-factors is connected. Thus $\mathcal{F}_q$ can only be an HB1F if $q\in \{2,5,11\}$ or $q=2^p$ for some odd prime $p$. The 1-factorisations $\mathcal{F}_5, \mathcal{F}_8, \mathcal{F}_{11}$, and $\mathcal{F}_{32}$ were shown to be HB1Fs in \cite{mitchell2022factorisation}, and the 1-factorisation $\mathcal{F}_2$ is trivially an HB1F. We have also shown computationally that $\mathcal{F}_{128}$ is an HB1F, which leads us to the following conjecture.

\begin{conjecture} $\mathcal{F}_{q}$ is a Hamilton-Berge 1-factorisation if and only if it is a connected 1-factorisation.
\end{conjecture}
\medskip
\medskip
\medskip
\noindent{\bf Acknowledgements}
The authors acknowledge the support of an Australian Government Research Training Program Scholarship, and the support of ARC grant DE200101802.

\bibliographystyle{abbrv}
% \bibliography{references}

\end{document}